\documentclass[reqno,twoside]{amsart} 
\usepackage{amsmath,amsthm,amssymb,amsfonts}
\usepackage{latexsym,esint}

\usepackage{epsfig,epic}

\usepackage{color}

\theoremstyle{plain}
\begingroup
\newtheorem{thm}{Theorem}[section]

\endgroup

\theoremstyle{definition}
\begingroup
\newtheorem{defn}[thm]{Definition}

\newtheorem{rem}[thm]{Remark}
\endgroup

\numberwithin{equation}{section}

\newcommand{\R}{\mathbb R} 

\newcommand{\leb}{{\mathcal L}}
\newcommand{\mtwo}{{\mathbb M}^{2{\times}2}}
\newcommand{\mthree}{{\mathbb M}^{3{\times}3}}

\newcommand{\dist}{{\rm dist}}
\newcommand{\sym}{{\rm sym}\,}

\newcommand{\wto}{\rightharpoonup}
\newcommand{\eps}{\varepsilon}
\newcommand{\Id}{{\rm Id}}

\newcommand{\E}{{\mathcal E}}
\newcommand{\F}{{\mathcal F}}
\newcommand{\J}{{\mathcal J}}
\newcommand{\IvK}{{\mathcal I}_{\rm vK}}
\newcommand{\JvK}{{\mathcal J}_{\rm vK}}
\newcommand{\Ib}{{\mathcal I}_\beta}
\newcommand{\Ilin}{{\mathcal I}_{\rm lin}}
\newcommand{\Jlin}{{\mathcal J}_{\rm lin}}

\begin{document}
 
\title[Convergence of equilibria of thin elastic plates]{Convergence of equilibria of 
thin elastic plates under physical growth conditions for the energy density}
\author[M.G. Mora]{Maria Giovanna Mora}
\author[L. Scardia]{Lucia Scardia}
\address[M.G. Mora]{Scuola Internazionale Superiore di Studi Avanzati, Via Beirut 2, 34014 Trieste (Italy)}
\email{mora@sissa.it}

\address[L. Scardia]{Max Planck Institute for Mathematics in the Sciences, Inselstra\ss e 22-26, 
04103 Leipzig (Germany)}
\email{scardia@mis.mpg.de}

\date{\today}


\begin{abstract} 
The asymptotic behaviour of the equilibrium configurations of a thin elastic plate
is studied, as the thickness $h$ of the plate goes to zero. More precisely, 
it is shown that critical points of the nonlinear elastic functional
$\E^h$, whose energies (per unit thickness) are bounded by $Ch^4$, converge
to critical points of the $\Gamma$-limit of $h^{-4}\E^h$. This is proved under 
the physical assumption that the energy density $W(F)$ blows up as $\det F\to0$.
\end{abstract}

\maketitle

\section{Introduction}

A {\em thin plate} is a three-dimensional body, occupying in a reference configuration a region of the form
$\Omega_h:=S{\times}(-\frac{h}{2},\frac{h}{2})$, where the mid-surface $S$ is a bounded domain of $\R^2$ and 
the small parameter $h>0$ measures the thickness of the plate. 

The elastic behaviour of such bodies is classically described by means of two-dimensional models, 
which are easier to handle both from an analytical and a computational viewpoint 
than their three-dimensional counterparts. There exists a large variety of such 
theories in the literature (see \cite{Cia, Lov} for a survey). However, as
their derivation is usually based on a priori assumptions on the form of relevant deformations,
their rigorous range of validity is typically not clear.
A fundamental question in elasticity is thus to justify rigorously 
lower dimensional models in relation to the three-dimensional theory. 

Recently, a novel variational approach through $\Gamma$-convergence has led to the rigorous derivation 
of a hierarchy of limiting theories. Among other features, it ensures the convergence of three-dimensional minimizers 
to minimizers of suitable lower dimensional limit energies.

In this paper we discuss the convergence of (possibly non-minimizing) {\em stationary points} of the
three-dimensional elastic energy, assuming {\em physical growth conditions} for the energy density.

\

We first review the main results of the variational approach.
Given a thin plate $\Omega_h$, the starting point of the variational analysis is 
the three-dimensional nonlinear elastic energy (scaled by unit thickness) $\E^h(w,\Omega_h)$ associated to
a deformation $w$ of the plate. The limiting behaviour of $\E^h$ as the thickness
of the plate tends to zero, can be described by the $\Gamma$-limit $\Ib$ of the 
functionals
$$
h^{-\beta}\E^h(\cdot,\Omega_h),
$$
as $h\to0$, for a given scaling $\beta\geq0$. As mentioned above, this implies, roughly speaking,
convergence of minimizers $w^h$ of $\E^h(\cdot,\Omega_h)$ (subject to applied forces or boundary conditions)
to minimizers of the two-dimensional energy $\Ib$, provided $\E^h(w^h,\Omega_h)\leq Ch^\beta$.
For the definition and main properties of $\Gamma$-convergence
we refer to the monographs \cite{Bra,DM}.

In this setting, $\Gamma$-convergence was first proved by Le Dret and Raoult in \cite{LDR} for the scaling $\beta=0$. 
This led to a rigorous justification of the {\em nonlinear membrane theory}. 
In the seminal papers \cite{FJM02,FJM06} Friesecke, James, and M\"uller established $\Gamma$-convergence for all $\beta\geq2$. The scaling $\beta=2$ corresponds in the limit to the {\em Kirchhoff plate theory}, while $\beta=4$
to the {\em von K\'arm\'an plate theory}. For $\beta>4$ the usual linear theory is derived, while the
intermediate scalings $2<\beta<4$ relate to a linear theory with constraints.
The case of $0<\beta< 5/3$ was recently solved by Conti and Maggi~\cite{CM}. 
The regime $5/3\leq\beta<2$ remains open and is conjectured to be relevant for crumpling of elastic sheets.
Analogous results have been proved for thin rods in~\cite{ABP, MM03, MM04}.

\

The intent of this paper is to investigate the convergence of stationary points of the three-dimensional
nonlinear elastic energy (subject to applied forces and boundary conditions) 
to stationary points of the $\Gamma$-limit functional. 
The first result concerning convergence of equilibria for thin bodies has been shown in \cite{MMS06}, in the case
of a thin strip and for the scaling $\beta=2$. 
This work has been extended in \cite{MM06} to the case of a thin rod in the regime $\beta=2$,
and then in \cite{MP06} to a thin plate in the von K\'arm\'an regime $\beta=4$ 
(see also \cite{LM08} for an extension to thin shells).
A crucial assumption in all these papers is that the elastic energy density $W$ is differentiable everywhere and its derivative satisfies a linear growth condition. 
This bound on $DW$ is 
unsatisfactory, as it prevents the blow-up of $W(F)$ as the determinant of the deformation gradient $F$ tends to zero (which corresponds to a strong compression of the body). We point out that, instead, the results in \cite{FJM02,FJM06}, as well as the ones in \cite{MM03, MM04}, do not require any bound from above on~$W$.
On the other hand, without assuming a linear growth condition on $DW$, it
is not even clear to which extent minimizers satisfy the Euler-Lagrange equations in the conventional form
(see \eqref{EL-class} below). 

A growth condition on $W$, which is compatible with the blow-up condition as 
$\det F\to0$ is:
\begin{equation}\label{intro1}
\big|DW(F)F^T \big| \leq k(W(F) + 1)  
\end{equation}
for every $F$ with $\det F>0$. In \cite{Ball83, Ball} Ball has shown that, under assumption \eqref{intro1}, it
is possible to derive an alternative first-order necessary condition for minimizers 
(Theorem~\ref{ball-thm}).
When minimizers are invertible, this condition corresponds to the equilibrium equation
for the Cauchy stress tensor (Remark~\ref{rem2}).

\

In this paper we focus on the scaling $\beta=4$ and we extend the analysis of \cite{MP06} 
to the case of an elastic energy density $W$ satisfying the physical growth condition \eqref{intro1}. 
More precisely, we call a deformation a stationary point of the three-dimensional energy 
if it satisfies the first-order necessary condition introduced by Ball in \cite{Ball83, Ball} (Definition~\ref{ourEq}).
In Theorem~\ref{mainthm} we prove that 
any sequence of stationary points $w^h$ of the three-dimensional energy,
satisfying $\E^h(w^h,\Omega_h)\leq Ch^4$, converges 
to a stationary point of the von K\'arm\'an functional (i.e., to a solution of the classical Euler-Lagrange equations of the limiting functional). 
This is the first result of convergence of equilibria
for thin plates compatible with the physical requirement that $W(F)\to+\infty$ as $\det F\to0$.

A key ingredient in the proof of our main result is the quantitative rigidity estimate proved by Friesecke, James, and M\"uller in \cite[Theorem~3.1]{FJM02}. It is first used 
to deduce compactness of sequences of stationary points from the bound on the elastic part of the energy,
and then to define suitable strain-like and stress-like variables $G^h$ and $E^h$ (see \eqref{dec} and \eqref{Eh}).
While in \cite{MP06} the $L^2$ bound on the strains $G^h$ (which follows directly from the rigidity estimate) implies immediately an analogous bound for the stresses $E^h$, the present case exhibits an additional difficulty. Indeed, in our setting the stresses $E^h$ turn out to be naturally defined as
\begin{equation}\label{intro2}
E^h=\frac{1}{h^2}DW(\Id+h^2G^h)(\Id+h^2G^h)^T
\end{equation}
($E^h$ can be interpreted as a sort of Cauchy stress tensor, read in the undeformed
configuration, see also Remark~\ref{rem2}). Hence, using the growth condition \eqref{intro1} and the bound on the elastic energy we 
can only deduce weak compactness of $E^h$ in $L^1$ and this convergence is not enough to pass to the limit
in the three-dimensional Euler-Lagrange equations (see Steps~2~--~3 of the proof
and the discussion therein).

This difficulty is overcome by identifying a sequence of measurable sets $B_h$,
which converge in measure to the whole set $\Omega:=S{\times}(-\frac12,\frac12)$
and satisfy the following properties.
On $B_h$ the remainder in the first order Taylor expansion of $DW$ around the identity is uniformly controlled with respect to $h$, so that one can
deduce an $L^2$ bound for $E^h$ from \eqref{intro2} and from the $L^2$ bound on $G^h$.
On the complement of $B_h$ one can use the growth condition \eqref{intro1} to show that the contribution of $E^h$ on this set is negligible at the limit in the $L^1$ norm.
This mixed type of convergence of the stresses is then shown to be sufficient to
pass to the limit in the three-dimensional Euler-Lagrange equations.

\

Our main result also applies to the regime $\beta>4$. More precisely, in Theorem~\ref{mainthm}
we also show that any sequence of stationary points $w^h$ of the three-dimensional energy,
satisfying $\E^h(w^h,\Omega_h)\leq Ch^\beta$ with $\beta>4$, converges 
to a stationary point of the functional of the linear plate theory. 

Convergence results for thin plates in the Kirchhoff regime $\beta=2$
and in the intermediate scalings $2<\beta<4$ are still open, even under
the simplifying assumption of linear growth of $DW$. 
The additional difficulties in the analysis of these regimes are due to the
weaker compactness properties arising from the rigidity theorem and
to the presence, in the limiting model, of a nonlinear or geometrically linear 
isometry constraint. 

\

The plan of the paper is the following. In Section~2 we describe the setting 
of the problem and we discuss the first order necessary condition by Ball. 
Section~3 contains the statement of the main result, which is proved in Section~4.

\section{Setting of the problem}

We consider a thin plate, whose reference configuration is given by the set 
$\Omega_h= S{\times}(-\frac{h}{2}, \frac{h}{2})$, where $S\subset\R^2$ is a bounded domain 
with Lipschitz boundary and $h>0$. 

Deformations of the plate are described by maps $w\colon\Omega_h\to\R^3$, which are assumed
to belong to the space $H^1(\Omega_h;\R^3)$. Moreover, we require the deformations $w$ 
to satisfy the boundary condition 
\begin{equation}\label{clamp}
w(z) = z \quad \mbox{for every } z\in\Gamma{\times}(-\tfrac{h}{2}, \tfrac{h}{2}),
\end{equation}
where $\Gamma$ is a (non-empty) relatively open subset of $\partial S$.

To any deformation $w\in H^1(\Omega_h;\R^3)$ we associate the total energy (per unit thickness)  
defined as 
\begin{equation}\label{elenergy}
\F^h(w)= \frac{1}{h} \int_{\Omega_h} W(\nabla w)\,dz -\frac{1}{h}\int_{\Omega_h} f^h{\,\cdot\,}w\, dz,
\end{equation} 
where $f^h\in L^2(\Omega_h;\R^3)$ is the density of a body force applied to $\Omega_h$.

On the stored-energy density $W\colon\mthree\to [0,+\infty]$ we require the following
asssumptions:
\begin{align}
& W\text{ is of class } C^1 \text{ on } \mthree_+;
\label{eq22}
\\ 
& W(F) = +\infty \quad \text{if }  \det F \leq 0, \qquad W(F)\to +\infty\quad \text{as } \det F \to 0^+;
\label{eq24}
\\
& W(RF) = W(F) \quad \mbox{for every } R\in SO(3),\ F\in\mthree \quad \mbox{(frame indifference)}.
\label{eq25}
\end{align}
Here $\mthree_+$ denotes the set of matrices $F\in\mthree$ with $\det F>0$, while 
$SO(3)$ denotes the set of proper rotations $\{R\in\mthree:\ R^TR=\Id,\ \det R=1\}$.
Condition \eqref{eq24} is related to the physical requirements of non-interpenetration of matter and preservation of orientation. It ensures local invertibility of $C^1$ deformations with finite energy.

We also require $W$ to have a single well at $SO(3)$, namely
\begin{align}
&\hspace{-4.7cm} W=0 \quad \mbox{on } SO(3);
\label{eq26}
\\
&\hspace{-4.7cm} W(F) \geq C\dist^2(F,SO(3));
\label{eq27}
\\
&\hspace{-4.7cm} W \,\, \mbox{is of class } C^2 \, \mbox{in a } \delta\mbox{-neighbourhood of } SO(3).
\label{eq28}
\end{align}

Finally, we assume the following growth condition:
\begin{equation}\label{Ball}
\big|DW(F)F^T \big| \leq k(W(F) + 1)  \quad \mbox{for every } F \in \mthree_+. 
\end{equation}
This is a mild growth condition on $W$, introduced by Ball in \cite{Ball83, Ball},
which is compatible with the physical requirement \eqref{eq24}, but is nevertheless 
sufficient to derive a first-order condition for minimizers of $\F^h$. 
In fact, by performing external variations $w+\eps\phi$ of a minimizer $w$,
one is formally led to the Euler-Lagrange equations in the conventional form
\begin{equation}\label{EL-class}
\int_{\Omega_h} DW(\nabla w){\,\cdot\,}\nabla\phi\,dz = \int_{\Omega_h} f^h{\,\cdot\,}\phi\, dz
\quad\forall \phi \text{ smooth with }\phi|_{\Gamma{\times}(-\frac{h}{2},\frac{h}{2})}=0.
\end{equation}
To justify rigorously this derivation, one has to require that either $DW$ is Lipschitz continuous
or the minimizer $w$ belongs to $W^{1,\infty}$ and satisfies a stronger orientation preserving condition,
namely $\det\nabla w\geq c>0$ a.e.\ in $\Omega_h$. However, none of these assumptions is satisfactory: 
the Lipschitz continuity of $DW$ is incompatible with \eqref{eq24}, while 
there may exist minimizers that do not belong to $W^{1,\infty}$ or do not satisfy the stronger orientation
preserving condition (see the discussion in \cite[Section~2.4]{Ball}).

If instead condition \eqref{Ball} is assumed, then it is possible to derive an alternative
equilibrium equation for minimizers.
More precisely, by considering variations of the form $w+\eps\,\phi{\,\circ\,}w$ one
can deduce the following condition.

\begin{thm}[{\cite[Theorem 2.4]{Ball}}] \label{ball-thm}
Assume that $W$ satisfies \eqref{eq22}, \eqref{eq24}, and \eqref{Ball}.
Let $U\subset\R^3$ be a bounded domain with a Lipschitz boundary
$\partial U=\partial U_1\cup\partial U_2\cup N$, where $\partial U_1$, $\partial U_2$ are
disjoint relatively open subsets of $\partial U$ and $N$ has 
zero two-dimensional measure.
Let $\bar w\in H^{1/2}(\partial U;\R^3)$ and $f\in L^2(U;\R^3)$.
Let $w\in H^1(U;\R^3)$ be a local minimizer of the functional
$$
\F(w):=\int_U W(\nabla w)\, dz -\int_U f{\,\cdot\,}w\, dz
$$
subject to the boundary condition $w=\bar w$ on $\partial U_1$, that is, there exists
$\eps>0$ such that $\F(w)\leq \F(v)$ for every $v\in H^1(U;\R^3)$ satisfying $\|v-w\|_{H^1}\leq\eps$
and $v=\bar w$ on $\partial U_1$.
Then
\begin{equation}\label{Euler-eq}
\int_U  DW(\nabla w)(\nabla w)^T {\,:\,} \nabla\phi(w)\, dz = 
\int_U f{\,\cdot\,} \phi(w)\,dz
\end{equation} 
for all $\phi\in C_b^1(\R^3;\R^3)$ such that $\phi\circ w= 0$ 
on $\partial U_1$ in the sense of trace.
\end{thm}

In the theorem above and in what follows, given a subset $U$ of $\R^n$ 
we denote by $C^k_b(U)$ the space of functions of class $C^k$ that
 are bounded in $U$, with bounded derivatives up to the $k$-th order.  
We also stress that in \eqref{Euler-eq} the term $\nabla\phi(w)$
denotes the gradient of $\phi$ computed at the point $w(z)$.

\begin{rem}\label{rem2}
Under the assumptions of Theorem~\ref{ball-thm},
if in addition $w$ is a smooth homeomorphism of $U$ onto $U':=w(U)$, then
equation \eqref{Euler-eq} reduces by means of a change of variables to 
$$
\int_{w(U)} T(w^{-1}(x)){\,:\,}\nabla\phi(x)\, dx=
\int_{w(U)} \tilde f(w^{-1}(x)){\,\cdot\,} \phi(x)\, dx
$$
for all $\phi\in C^1(\R^3;\R^3)$ such that $\phi|_{w(\partial U_1)}= 0$. In the formula above $T$ is the Cauchy stress tensor:
$$
T(z)=(\det\nabla w(z))^{-1}DW(\nabla w(z))(\nabla w(z))^T, \qquad z\in U
$$
and $\tilde f=(\det\nabla w)^{-1}f$
(see \cite[Theorem~2.6]{Ball}). In other words, Theorem~\ref{ball-thm} asserts
that the equilibrium equations are satisfied in the deformed configuration.
\end{rem}

In our setting it is natural to assume \eqref{Euler-eq} as definition of
stationary points of $\F^h$. Our aim is to analyse their limit behaviour, 
as the thickness $h$ goes to~$0$. 
To do so, it is convenient to perform a change of
variables and to reduce to a fixed domain independent of $h$. 
Thus, we consider the scaling $(z',z_3) = (x', hx_3)$, 
$\nabla_h = \left(\nabla',\frac{1}{h}\partial_{3}\right)$, 
$y(x) = w(z)$, and $g^h(x) = f^h(z)$, and we introduce the functional 
\begin{equation}\label{riscalenergy}
\J^h(y) = \F^h(w) = \int_{\Omega}W(\nabla_h y) \,dx -\int_\Omega g^h{\,\cdot\,}y\, dx,
\end{equation}
where $\Omega = S{\times}(-\frac{1}{2}, \frac{1}{2})$ and the scaled deformation 
$y\in H^1(\Omega;\R^3)$ satisfies the boundary condition
\begin{equation}\label{yclamp}
y(x) = (x',hx_3) \quad \mbox{for every } x=(x',x_3)\in\Gamma{\times}(-\tfrac12, \tfrac12).
\end{equation}

According to Theorem~\ref{ball-thm}, we give the following definition.

\begin{defn}\label{ourEq}
We say that a deformation $y\in H^1(\Omega;\R^3)$ is a stationary point of $\J^h$, 
subject to clamped boundary conditions on $\Gamma{\times}(-\frac12,\frac12)$, 
if $y(x) = (x', hx_3)$ for every $x\in \Gamma{\times}(-\frac12,\frac12)$ 
and the following equation is satisfied:
$$
\int_{\Omega} DW(\nabla_h y)(\nabla_h y)^T {\,:\,} \nabla\phi(y)\, dx = 
\int_{\Omega} g^h {\,\cdot\,} \phi(y)\,dx
$$
for all $\phi \in C_b^1(\R^3;\R^3)$ satisfying $\phi(x',hx_3) = 0$ for every  
$x \in \Gamma{\times}(-\frac12,\frac12)$.
\end{defn}

\begin{rem}
In \cite{Ball} Ball has shown that, if $W$ satisfies the growth condition
$$
\big|F^TDW(F)\big| \leq k(W(F) + 1)  \quad \mbox{for every } F \in \mthree_+ 
$$
(which implies, but is not equivalent to \eqref{Ball}, see \cite[Proposition~2.3]{Ball}),
then local minimizers of $\F$ satisfy the equation
$$
\int_U  \big( W(\nabla w)\,\Id - (\nabla w)^T\!DW(\nabla w)\big){\,:\,} \nabla\phi\, dz = 
\int_U (\nabla w)^T\!f{\,\cdot\,} \phi\,dz
$$
for all $\phi\in C^1_0(U;\R^3)$. This equation is obtained by performing internal variations of the form
$w\circ\psi_\eps$, with $\psi_\eps^{-1}(x)=x+\eps\phi(x)$, and can be viewed
as a multi-dimensional version of the classical
Du Bois-Raymond equation of the one-dimensional calculus of variations.
To the purpose of our analysis the use of this equilibrium equation in place of \eqref{Euler-eq} seems
to be less convenient. Indeed, the requirement of zero boundary values for the test functions 
suggests that the equation does not provide precise information about the boundary behaviour of the limiting quantities. 
Moreover, it imposes a severe restriction on the choice of admissible test functions. 
\end{rem}

\begin{rem}
If $W$ satisfies \eqref{Ball}, then $W$ has polynomial growth, that is,
there exists $s>0$ such that
$$
W(F)\leq C(|F|^s+|F^{-1}|^s) \quad \text{for all } F\in\mthree_+
$$ 
(see \cite[Proposition~2.7]{Ball}). In particular, examples of functions satisfying \eqref{eq22}--\eqref{Ball} are:
$$
W(F) = |(F^T F)^{1/2} - \Id |^2 + |\log \det F|^p \qquad \mbox{ for } F\in \mathbb{M}_+^{3\times 3},
$$
or
$$
W(F) = |(F^T F)^{1/2} - \Id |^2 + \Big|\frac{1}{\det F}-1\Big|^p \qquad
\mbox{for } F\in \mathbb{M}_+^{3\times 3},
$$
where $p>1$ and $W$ is intended to be $+\infty$ if $\det F\leq0$.
\end{rem}

\section{Statement of the main result}

In this paper we focus on the asymptotic study of stationary points $y^h$ of $\J^h$ 
(according to Definition~\ref{ourEq}) with elastic energy (per unit thickness) of order $h^\beta$ with $\beta\geq4$, that is,
\begin{equation}\label{el-bd}
\int_\Omega W(\nabla_h y^h)\, dx\leq Ch^\beta, \qquad \beta\geq4.
\end{equation}
For simplicity we assume that the body forces $g^h$ are independent of the variable $x_3$ and 
normal to the mid-surface of the plate; more precisely, we assume $g^h(x)=h^{(\beta+2)/2}g(x')e_3$, 
where $g\in L^2(S)$ is given. The scaling $h^{(\beta+2)/2}$ of the normal force ensures 
consistency with the elastic energy scaling \eqref{el-bd}.

In \cite{FJM06} Friesecke, James, and M\"uller have identified the limit
of the functionals $h^{-\beta}\J^h$, in the sense of $\Gamma$-convergence, 
under the assumptions \eqref{eq25}--\eqref{eq28}.
For $\beta=4$ the $\Gamma$-limit $\JvK$ can be expressed in terms of the 
averaged in-plane and out-of-plane displacements $u$ and $v$ (see \eqref{def-uv}) 
and is given by
\begin{equation}\label{J}
\JvK(u,v) = \IvK(u,v) -\int_S gv\, dx',
\end{equation}
where, for $u\in H^1(S;\mathbb{R}^2)$ and $v\in H^2(S)$, the von K\'arm\'an functional 
$\IvK$ is defined as 
\begin{equation}\label{vK}
\IvK(u,v) = \frac12 \int_{S} Q_2 \Big(\sym\nabla'u +\frac12\nabla'v{\,\otimes\,} \nabla'v\Big)\, dx' 
+ \frac{1}{24} \int_S Q_2 ((\nabla')^2 v)\, dx'.
\end{equation}
Here $Q_2$ is a quadratic form that can be computed from the linearization 
of $W$ at the identity. More precisely, we consider the quadratic form 
$Q_3(F) = D^2W(\Id)F{\,:\,}F$ on $\mthree$ and define the quadratic form $Q_2$ on 
$\mtwo$ through the following minimization procedure:
\begin{equation}\label{defQ2}
Q_2(G) = \leb_2 G{\,:\,}G := \min_{F''=G} Q_3(F),
\end{equation} 
where $F''$ denotes the $2{\times}2$ submatrix given by $F''_{ij} = F_{ij}$, 
$1\leq i,j \leq 2$.

For $\beta>4$ the $\Gamma$-limit $\Jlin$ depends only on the averaged 
out-of-plane displacement $v$ and is given by
\begin{equation}\label{linJ}
\Jlin(v) = \Ilin(v) -\int_S gv\, dx',
\end{equation}
where $\Ilin$ is the functional of linear plate theory, defined as
\begin{equation}\label{linvK}
\Ilin(v) = \frac{1}{24} \int_S Q_2 ((\nabla')^2 v)\, dx'.
\end{equation}
for every $v\in H^2(S)$. 

The $\Gamma$-convergence result guarantees, in particular, that given a minimizing sequence $y^h$
satisfying
$$
\limsup_{h\to0} \frac{1}{h^\beta}\big(\J^h(y^h)-\inf \J^h \big)=0,
$$
the averaged in-plane and out-of-plane displacements associated with $y^h$ converge to a minimizer
$(u,v)$ of $\JvK$ if $\beta=4$. If $\beta>4$, they converge to a pair of the form
$(0,v)$, where $v$ is a minimizer of $\Jlin$.

\

To set the stage for our result on the convergence of equilibria, we derive the
Euler-Lagrange equations for a minimizer $(u,v)$ of $\JvK$.
First of all, from the clamped boundary conditions \eqref{yclamp} it follows that the limiting displacement $(u,v)$ satisfies
\begin{equation}\label{bcuv}
u(x') = 0 \quad \text{and} \quad v(x') = 0, \ \nabla'v(x') = 0 \quad \mbox{for every } x' \in \Gamma.
\end{equation}
By computing the respective variations of $\JvK$ in $u$ and $v$ we obtain the following Euler-Lagrange equations in weak form:
\begin{equation}\label{EulLag-u}
\begin{split}
\int_S \Big(\leb_2\big(\sym \nabla'u +\frac12 
\nabla'v &\otimes \nabla'v\big){\,:\,}(\nabla'v \otimes \nabla'\varphi)\\
&{} + \frac{1}{12}\leb_2 ((\nabla')^2 v){\,:\,}(\nabla')^2\varphi 
- g\varphi\Big)\, dx' = 0 
\end{split}
\end{equation}
for every $\varphi \in H^2(S)$ with $\varphi|_\Gamma=0$, $\nabla'\varphi|_\Gamma = 0$, and
\begin{equation}\label{EulLag-v}
\int_S \leb_2 \big(\sym\nabla'u + \frac12\nabla'v \otimes \nabla'v\big) {\,:\,} \nabla'\psi \, dx' = 0 
\end{equation}
for every $\psi\in H^1(S;\R^2)$ with $\psi|_\Gamma=0$.

In the case of the linear functional $\Jlin$ the limit displacement $v$ satisfies the boundary conditions
\begin{equation}\label{bcvlin}
v(x') = 0, \quad \nabla'v(x') = 0 \quad \mbox{for every } x' \in \Gamma
\end{equation}
and the Euler-Lagrange equations are given by
\begin{equation}\label{EulLag-lin}
\int_S \Big(\frac{1}{12}\leb_2 ((\nabla')^2 v){\,:\,}(\nabla')^2\varphi 
- g\varphi\Big)\, dx' = 0 
\end{equation}
for every $\varphi \in H^2(S)$ with $\varphi|_\Gamma=0$, $\nabla'\varphi|_\Gamma = 0$,

From now on we will adopt the notation $y=(y', y_3)$.

The main result of the paper is the following.

\begin{thm}\label{mainthm}
Assume that the energy density $W$ satisfies \eqref{eq22}--\,\eqref{Ball}. Let $\beta\geq4$.
Let $(y^h)$ be a sequence of stationary points of $\J^h$ according to Definition~\ref{ourEq}, 
with $g^h(x) = h^{(\beta+2)/2}g(x')e_3$.
Assume further that
\begin{equation}\label{growth}
\int_{\Omega}W(\nabla_h y^h) \,dx \leq C h^\beta.
\end{equation}
Set
\begin{equation}\label{def-uv}
\begin{array}{rcl}
u^h(x') & := & \displaystyle  
\frac{1}{h^{\beta/2}}\int_{-\frac12}^{\frac12} 
\big((y^h)'(x',x_3) - x'\big)\,dx_3, 
\smallskip
\\
v^h(x') & := & \displaystyle 
\frac{1}{h^{(\beta-2)/2}} \int_{-\frac12}^{\frac12} y_3^h(x',x_3)\, dx_3
\end{array}
\end{equation}
for every $x'\in S$.
Then the following assertions hold.
\smallskip

\begin{itemize}
 	\item[(i)] (von K\'arm\'an regime) Assume $\beta=4$. Then, 
there exist $u\in H^1(S;\mathbb{R}^2)$ and $v\in H^2(S)$ such that, up to subsequences,
\begin{equation}\label{convU}
u^h\wto u \quad \mbox{weakly in } H^1(S;\mathbb{R}^2)
\end{equation}
and 
\begin{equation}\label{convV}
v^h\to v \quad \mbox{strongly in } H^1(S),
\end{equation}
as $h\to0$, and the limit displacement $(u,v)$ solves \eqref{EulLag-u}--\,\eqref{EulLag-v}, 
and satisfies the boundary conditions \eqref{bcuv}.
\medskip
	\item[(ii)] (linear regime) Assume $\beta>4$. Then, \eqref{convU} and \eqref{convV} hold with
$u=0$, and the limit displacement $v$ solves \eqref{EulLag-lin} and satisfies the boundary conditions \eqref{bcvlin}. 
\end{itemize}
\end{thm}

\begin{rem}
If $y^h$ is a sequence of minimizers of $\J^h$ with $g^h(x) = h^{(\beta+2)/2}g(x')e_3$, then 
condition \eqref{growth} is automatically satisfied. 
This can be proved by means of a Poincar\'e-like inequality related to the rigidity theorem by Friesecke, James, and M\"uller (see the proof of \cite[Theorem~2, part~iii]{FJM06}).
\end{rem}

\begin{rem}
In \cite{Mie} Mielke used a centre manifold approach to compare solutions 
in a thin strip to a one-dimensional problem. 
This method works already for finite $h$, but it requires that the nonlinear strain $(\nabla_hy)^T\nabla_hy$ is close to the identity in $L^\infty$. 
Applied forces $g$ are also difficult to include.
We also mention a more recent result by Monneau \cite{Monneau}, based on 
a careful use of the implicit function theorem. He shows that, starting from a 
sufficiently smooth solution of the von K\'arm\'an equations, there exists a nearby 
solution of the three-dimensional problem.
\end{rem}

\section{Proof of Theorem~\ref{mainthm}}

This section is devoted to the proof of the main result of the paper.

\begin{proof}[Proof of Theorem~\ref{mainthm}]
Let $\beta\geq4$. For notational convenience we set 
$$
\alpha:=(\beta+2)/2,
$$
so that $\alpha\geq 3$.
Let $(y^h)$ be a sequence of stationary points of $\J^h$, i.e., suppose that 
\begin{equation}\label{EuleroBh}
\int_{\Omega} DW(\nabla_h y^h)(\nabla_h y^h)^T {\,:\,} \nabla\phi(y^h) \,dx = 
\int_{\Omega} h^\alpha g e_3 {\,\cdot\,} \phi(y^h)\,dx
\end{equation}
for all $\phi\in C^1_b(\R^3;\R^3)$ satisfying $\phi(x',hx_3)=0$ for every  
$x\in\Gamma{\times}(-\frac12,\frac12)$. 
{}Furthermore, assume that condition \eqref{growth} is fulfilled.

\medskip
\textit{Step 1. Decomposition of the deformation gradients in rotation and strain.}
The energy bound \eqref{growth} and the coercivity condition \eqref{eq27} imply that
$$
\int_{\Omega}\dist^2(\nabla_h y^h,SO(3)) \,dx \leq Ch^{2\alpha-2}.
$$
Owing to the rigidity estimate \cite[Theorem~3.1]{FJM02}, 
this bound guarantees the existence of a sequence of smooth rotations $R^h$, 
whose $L^2$ distance from $\nabla_h y^h$ is of order $h^{\alpha-1}$.
A careful analysis of the increment of $R^h$ in neighbouring squares of side $h$ shows that
the gradient of $R^h$ is well controlled in terms of $h$. From this it follows that 
$\nabla_h y^h$ must converge to a constant rotation (namely, the identity, because of the boundary condition) and that the in-plane and out-of-plane displacements satisfy 
the compactness properties \eqref{convU} and \eqref{convV}, respectively.
 
More precisely, arguing as in \cite[Theorem 6 and Lemma 1]{FJM06}, one can 
construct a sequence $(R^h)\subset C^{\infty}(S;\mthree)$ 
such that $R^h(x') \in SO(3)$ for every $x' \in S$ and 
\begin{align}
& \|\nabla_h y^h -R^h \|_{L^2} \leq Ch^{\alpha-1},
\label{boundyh} \\
& \|\nabla' R^h \|_{L^2} 
\leq Ch^{\alpha-2},
\label{id} \\
& \|R^h - \Id \|_{L^2} \leq Ch^{\alpha-2}.\label{id0} 
\end{align}
{}From \eqref{boundyh} and \eqref{id0} it follows that $\nabla_h y^h$ converge to $\Id$
strongly in $L^2(\Omega;\mthree)$; in particular, $\nabla y^h\to \hbox{diag}\{1,1,0\}$ strongly
in $L^2(\Omega;\mthree)$. Therefore, by the boundary condition 
$y^h(x',x_3)=(x',hx_3)$ for every $x\in\Gamma{\times}(-\frac12,\frac12)$ and  
the Poincar\'e inequality, we have that
\begin{equation}\label{yh-conv}
y^h\to (x',0) \quad \mbox{strongly in } H^1(\Omega;\R^3).
\end{equation}

By \cite[Lemma~1]{FJM06} there exist $u\in H^1(S;\R^2)$ and $v\in H^2(S)$ such that 
\eqref{convU} and \eqref{convV} hold true, up to subsequences. From the boundary condition
satisfied by $y^h$ we obtain immediately that $u(x')=0$ and $v(x')=0$ for every 
$x'\in\Gamma$.
Moreover, by \cite[Corollary~1]{FJM06} the first moment $\xi^h$ of the 
in-plane displacement satisfies
$$
\xi^h(x'):=\frac{1}{h^{\alpha-1}}\int_{-\frac12}^{\frac12} x_3\big( (y^h)'(x',x_3) - x'\big) dx_3 
\ \wto \ - \frac{1}{12} \nabla'v
\quad \mbox{weakly in } H^1(S;\mathbb{R}^2).
$$
As $\xi^h=0$ on $\Gamma$ for every $h$, this implies $\nabla'v=0$ on $\Gamma$.
Finally, \cite[Lemma~1]{FJM06} guarantees the following convergence properties for $R^h$:
\begin{equation}\label{RId}
A^h:=\frac{R^h - \Id}{h^{\alpha-2}} \ \wto \ A := -(\nabla'v,0)\otimes e_3 
+ e_3\otimes(\nabla'v,0) \quad \mbox{in } H^1(S;\mthree)
\end{equation}
and 
\begin{equation}\label{RId2}
\sym\frac{R^h - \Id}{h^{2\alpha-4}} \to \frac{A^2}{2}\quad \mbox{in } L^q(S;\mthree), \quad 
\forall q<\infty.
\end{equation}
In particular, by the Poincar\'e-Wirtinger inequality and the equations \eqref{boundyh} and \eqref{RId2}, 
we obtain
\begin{equation}\label{y2-est}
\big\|\frac{y^h_3}{h} -x_3-h^{\alpha-3}v^h\big\|_{L^2}\leq C\big\|\frac{\partial_3y^h_3}{h} -1\big\|_{L^2}
\leq Ch^{\alpha-1}.
\end{equation}

The bound \eqref{boundyh} suggests the following decomposition for the deformation gradients: 
\begin{equation}\label{dec}
\nabla_h y^h = R^h (\Id + h^{\alpha-1} G^h).
\end{equation}
By \eqref{boundyh} the $G^h\colon\Omega\to\mthree$ are bounded in 
$L^2(\Omega; \mthree)$. Thus, up to subsequences, 
$G^h\wto G$ weakly in $L^2(\Omega; \mthree)$ for some $G\in L^2(\Omega;\mthree)$.
By \cite[Lemma~2]{FJM06} the limiting strain $G$ satisfies
\begin{equation}\label{G-id1}
G''(x',x_3) =  G_0(x') - x_3 (\nabla')^2v, 
\end{equation}
where
\begin{eqnarray}
\label{G-id2}
\sym G_0 = \sym\nabla'u + \frac12\nabla'v \otimes \nabla'v &\quad \text{if } \alpha=3,
\\
\label{G-id2lin}
\sym G_0 = \sym\nabla'u &\quad \text{if } \alpha>3.
\end{eqnarray}
We recall that $G''$ denotes the $2{\times}2$ submatrix $G''_{ij}=G_{ij}$, 
$1\leq i,j\leq2$.

\medskip
\textit{Step 2. Estimate for the scaled stress.}
Let $E^h\colon\Omega\to\mthree$ be the scaled stress defined by
\begin{equation}\label{Eh}
E^h:= \frac{1}{h^{\alpha-1}} DW(\Id + h^{\alpha-1} G^h)(\Id + h^{\alpha-1} G^h)^T.
\end{equation} 
Notice that  $E^h$ is symmetric, due to the frame indifference of $W$.
Moreover, the following estimate holds true:
\begin{equation}\label{estEh}
|E^h| \leq C \Big(\frac{W(\Id + h^{\alpha-1} G^h)}{h^{\alpha-1}} + |G^h| \Big).
\end{equation}
Indeed, if $h^{\alpha-1}|G^h|\leq\delta/2$, where $\delta$ is the size of the neighbourhood in \eqref{eq28}, then 
$$
DW(\Id + h^{\alpha-1} G^h) = h^{\alpha-1} D^2W(F^h)G^h,
$$
for some matrix $F^h\in\mthree$ with $|F^h-\Id|\leq\delta/2$. As $D^2W$ is bounded
in this set, we deduce that
$$
|DW(\Id + h^{\alpha-1} G^h)|\leq Ch^{\alpha-1}|G^h|,
$$
which implies
$$
|E^h| \leq C|G^h| + Ch^{\alpha-1} |G^h|^2 \leq C(1+ \delta)|G^h|.
$$
If $h^{\alpha-1}|G^h|>\delta/2$, by \eqref{Ball} we have
$$
|E^h| \leq \frac{1}{h^{\alpha-1}} k \big(W(\Id + h^{\alpha-1} G^h) + 1\big) 
\leq k \frac{W(\Id + h^{\alpha-1} G^h)}{h^{\alpha-1}} + \frac{2k}{\delta}\, |G^h|.
$$
We notice that we are allowed to use the bound \eqref{Ball}, as $W(\nabla_h y^h)$ is finite a.e.\ in~$\Omega$
by \eqref{growth}, hence $\det\nabla_h y^h=\det(\Id + h^{\alpha-1} G^h)>0$ a.e.\ in $\Omega$.
This concludes the proof of (\ref{estEh}).

\medskip
\textit{Step 3. Convergenge properties of the scaled stress.}
By the decomposition (\ref{dec}) and the frame indifference of $W$, we obtain 
\begin{eqnarray*}
DW(\nabla_h y^h)(\nabla_h y^h)^T & = & R^h DW(\Id + h^{\alpha-1}G^h)(\Id + h^{\alpha-1} G^h)^T (R^h)^T 
\\
& = & h^{\alpha-1} R^h E^h (R^h)^T. 
\end{eqnarray*}
Thus, in terms of the stresses $E^h$ the Euler-Lagrange equations (\ref{EuleroBh}) can be written~as
\begin{equation}\label{EuleroBEh}
\int_{\Omega} R^h E^h (R^h)^T{\,:\,} \nabla\phi(y^h) \,dx = 
\int_{\Omega} h g e_3{\,\cdot\,} \phi(y^h)\,dx
\end{equation}
for all $\phi\in C_b^1(\R^3;\R^3)$ satisfying $\phi(x',hx_3)=0$ for every  
$x\in\Gamma{\times}(-\frac12,\frac12)$.

In order to pass to the limit in \eqref{EuleroBEh} we are interested in studying the convergence
properties of the scaled stresses $E^h$.

By \eqref{estEh}, \eqref{growth} and the fact that the $G^h$ are bounded in $L^2(\Omega;\mthree)$, 
we deduce that for every measurable set $\Lambda\subset\Omega$
\begin{eqnarray}
\int_{\Lambda} |E^h|\, dx & \leq &
C\int_{\Lambda} \frac{W(\Id + h^{\alpha-1} G^h)}{h^{\alpha-1}}\,dx + 
C \int_{\Lambda} |G^h|\,dx
\nonumber\\
& \leq & Ch^{\alpha-1} + C|\Lambda|^{1/2}. \label{estL1-v}
\end{eqnarray}
This bound ensures that the scaled stresses $E^h$ are bounded and equi-integrable
in $L^1(\Omega;\mthree)$. Therefore, by the Dunford-Pettis theorem 
\begin{equation}\label{conL1s}
E^h \wto E \quad\mbox{weakly in } L^1(\Omega;\mthree)
\end{equation}
for some $E\in L^1(\Omega;\mthree)$. 
In particular, since $E^h$ is symmetric, also $E$ is symmetric.

One can immediately realize that weak convergence of $E^h$ in $L^1$ is not enough to pass to the limit
in \eqref{EuleroBEh}. This is due to the fact that, for instance, one cannot guarantee uniform convergence of 
the term $\nabla\phi(y^h)$ (recall that for $y^h$ we have the convergence \eqref{yh-conv}).
Therefore, some more refined convergence properties for $E^h$ are needed. In particular,
we shall identify a sequence of sets $B_h$, whose measures converge to the measure of $\Omega$
(and therefore, on $\Omega\setminus B_h$ the sequence $E^h$ converges to $0$ strongly in $L^1$  
by \eqref{estL1-v}), and such that on $B_h$ the sequence $E^h$ is weakly compact in  $L^2$.
Using the $C^1_b$ regularity of test functions, we shall show that this mixed type of convergence
is sufficient to derive the limit equations.
  
Let $B_h:= \{x\in\Omega: h^{\alpha-1-\gamma}|G^h(x)|\leq 1\}$, with $\gamma\in(0,\alpha-2)$, and
let $\chi_h$ be its characteristic function. Notice that 
$$
|\Omega\setminus B_h|\leq \int_{\Omega\setminus B_h} h^{\alpha-1-\gamma}|G^h|\,dx 
\leq Ch^{\alpha-1-\gamma} |\Omega\setminus B_h|^{1/2}\|G^h\|_{L^2},
$$
hence
\begin{equation}\label{mOh}
|\Omega\setminus B_h|\leq Ch^{2(\alpha-1-\gamma)}.
\end{equation}
This implies in particular that $\chi_h$ converges to $1$ in measure and thus, $\chi_h G^h$
converges to $G$ weakly in $L^2(\Omega;\mthree)$.

From \eqref{estL1-v} and \eqref{mOh} it follows that
\begin{equation}\label{estL1}
\int_{\Omega\setminus B_h}|E^h|\, dx \leq Ch^{\alpha-1-\gamma},
\end{equation}
hence
\begin{equation}\label{conL1}
(1-\chi_h)E^h \to 0 \quad \mbox{strongly in }L^1(\Omega;\mthree).
\end{equation}

On the set $B_h$ we have a uniform control of the term $h^{\alpha-1}G^h$, 
so that we can deduce weak convergence of $\chi_hE^h$ in $L^2(\Omega;\mthree)$ from
the weak convergence of $G^h$ simply by Taylor expansion.
More precisely, let $\leb$ be the linear operator defined by $\leb := D^2W(\Id)$. We claim that 
\begin{equation}\label{conL2}
\chi_h E^h \wto \leb G \quad \mbox{weakly in } 
L^2(\Omega; \mthree).
\end{equation}
We note that, as $R^h$ converges boundedly in measure to $\Id$, the claim implies also that
$\chi_hR^hE^h$ converges to $\leb G$ weakly in $L^2(\Omega; \mthree)$. This remark will be
repeatedly used in the next steps of the proof.

By Taylor expansion we have 
$$
DW(\Id+h^{\alpha-1}G^h)= h^{\alpha-1} \leb G^h +\eta(h^{\alpha-1}G^h),
$$
where the remainder $\eta$ satisfies $\eta(F)/|F|\to 0$, as $|F|\to 0$.
This identity leads to the following decomposition of $\chi_h E^h$:
\begin{multline}\label{3terms}
\chi_h E^h = 
\chi_h \frac{1}{h^{\alpha-1}} \big(h^{\alpha-1}\leb G^h + 
\eta(h^{\alpha-1} G^h) \big)(\Id + h^{\alpha-1}G^h)^T 
\\
= \chi_h \leb G^h + \chi_h h^{\alpha-1} \leb G^h(G^h)^T + 
\chi_h \frac{\eta(h^{\alpha-1} G^h)}{h^{\alpha-1}} + \chi_h \eta(h^{\alpha-1} G^h)(G^h)^T.
\end{multline}
To prove the claim \eqref{conL2} we 
analyse carefully each term on the right-hand side of (\ref{3terms}). 
The weak convergence of $\chi_h G^h$ to $G$ in $L^2(\Omega; \mthree)$ and 
the linearity of $\leb$ yield
\begin{equation}\label{first}
\leb(\chi_h G^h)\wto\leb G \quad \mbox{weakly in } 
L^2(\Omega; \mthree). 
\end{equation}
The second term in the right-hand side of \eqref{3terms} can be estimated as follows:
$$
|\chi_h h^{\alpha-1}\leb G^h (G^h)^T| \leq \chi_h Ch^{\alpha-1} |G^h|^2 \leq Ch^\gamma|G^h|.
$$
Therefore, it converges to zero strongly in $L^2(\Omega;\mthree)$ by the
$L^2$ bound of the $G^h$. 
As for the third term in \eqref{3terms}, we have the following bound:
$$
\left|\chi_h \frac{\eta(h^{\alpha-1} G^h)}{h^{\alpha-1}} \right| 
\leq \omega(h^\gamma)\, |G^h|,
$$
where for every $t>0$ we have set
$$
\omega(t):= \sup \left\{\frac{|\eta(A)|}{|A|}: |A| \leq t \right\}.
$$
Since $\omega(t)\to 0$ for $t\to 0^+$, we can conclude as before that
$\chi_h\eta(h^{\alpha-1} G^h)/h^{\alpha-1}$ converges to zero strongly in $L^2(\Omega;\mthree)$. Finally, as 
$$
|\chi_h \eta(h^{\alpha-1} G^h)(G^h)^T| 
 \leq  h^{\alpha-1}\chi_h \omega(h^\gamma)\, |G^h|^2 \leq \omega(h^\gamma)\, h^\gamma |G^h|,
$$
also this last term converges to zero strongly in $L^2(\Omega;\mthree)$.
Combining together \eqref{first} and the previous convergence properties, 
we obtain the claim (\ref{conL2}). 
Notice that by \eqref{conL1s} and \eqref{conL1} this implies $E =\leb G\in L^2(\Omega;\mthree)$.

\medskip
\textit{Step 4. Consequences of the Euler-Lagrange equations.}
We now begin to derive some preliminary information from the Euler-Lagrange equations \eqref{EuleroBEh}.

Let $\phi\in C_b^1(\R^3;\R^3)$ be such that $\phi(x',x_3) = 0$ 
for every $x \in \Gamma{\times}(-\frac12,\frac12)$, and let
us consider a test function of the form $\phi^h(x):=h\phi(x', \frac{x_3}{h})$.
We notice that $\phi^h$ is an admissible test function, as $\phi^h\in C_b^1(\R^3;\R^3)$ and 
$\phi^h(x',hx_3)=h\phi(x',x_3) = 0$ for every $x \in \Gamma {\times}(-\frac12,\frac12)$.

Inserting $\phi^h$ in \eqref{EuleroBEh} leads to 
$$
\begin{array}{l}
\displaystyle 
h\int_{\Omega}\sum_{i=1}^2 R^h E^h (R^h)^T e_i{\,\cdot\,}\partial_i
\phi\big((y^h)', \frac{y^h_3}{h}\big)\, dx 
\\
\displaystyle
{}+ \int_{\Omega} R^h E^h (R^h)^Te_3 {\,\cdot\,} \partial_3 \phi \big((y^h)', 
\frac{y^h_3}{h}\big)\,dx
=\int_{\Omega} h^2 g e_3{\,\cdot\,}\phi\big((y^h)', \frac{y^h_3}{h}\big) \,dx.
\end{array}
$$
As $R^hE^h(R^h)^T$ is bounded in $L^1(\Omega;\mthree)$ and $\nabla'\phi$ is a bounded function, 
the first integral on the left-hand side converges to zero as $h\to0$. 
Since the right-hand side is clearly infinitesimal, we deduce
\begin{equation}\label{limit}
\lim_{h\to0} \int_{\Omega} R^hE^h(R^h)^Te_3{\,\cdot\,}\partial_3\phi\big((y^h)', 
\frac{y^h_3}{h}\big) \,dx = 0.
\end{equation}
On the other hand, owing to \eqref{convV}, \eqref{yh-conv}, \eqref{y2-est},  
and to the continuity and boundedness of $\partial_3\phi$, we have 
\begin{align}
\displaystyle\label{convyv}
\partial_3\phi\big((y^h)', \frac{y^h_3}{h}\big) \ \to \ \partial_3\phi(x',x_3+v(x')) 
& \mbox{ \ strongly in } L^2(\Omega;\R^3), &  \text{ if }\alpha=3, 
\\
\displaystyle\label{convyvlin}
\partial_3\phi\big((y^h)', \frac{y^h_3}{h}\big) \ \to \ \partial_3\phi(x',x_3) 
&\mbox{ \ strongly in } L^2(\Omega;\R^3), & \text{ if }\alpha>3,
\end{align}
(the convergence is actually strong in $L^p(\Omega;\R^3)$ for every $p<\infty$).
Therefore, splitting the integral in \eqref{limit} as 
\begin{eqnarray*}
\lefteqn{\int_{\Omega} R^h E^h (R^h)^Te_3 {\,\cdot\,} \partial_3 \phi 
\big((y^h)', \frac{y^h_3}{h} \big) \,dx} 
\\
& = & \int_{\Omega} \chi_h R^h E^h (R^h)^Te_3 {\,\cdot\,} \partial_3 \phi 
\big((y^h)', \frac{y^h_3}{h} \big) \,dx \\
& & {}+ \int_{\Omega} (1 - \chi_h) R^h E^h (R^h)^Te_3 {\,\cdot\,} \partial_3 \phi 
\big((y^h)', \frac{y^h_3}{h} \big) \,dx
\end{eqnarray*}
and using \eqref{conL1} and \eqref{conL2}, we conclude that
\begin{eqnarray}
\label{E2-part}
\int_{\Omega} E e_3 {\,\cdot\,} \partial_3\phi(x',x_3+v(x')) \, dx = 0 & \qquad \text{if }\alpha=3, 
\\
\label{E2-partlin}
\int_{\Omega} E e_3 {\,\cdot\,} \partial_3\phi \, dx = 0 & \qquad \text{if }\alpha>3,
\end{eqnarray}
for every $\phi\in C_b^1(\R^3;\R^3)$ such that $\phi(x',x_3) = 0$ 
for every $x \in\Gamma {\times}(-\frac12,\frac12)$.

In the case $\alpha=3$, let $w_k\in C^1_b(\R^2)$ be a sequence of functions such that the restriction of $w_k$ to $S$ 
converges to $v$ strongly in $L^2(S)$ and $w_k(x') = 0$ for every $x'\in\Gamma$.
Then, given any $\phi\in C_b^1(\R^3;\R^3)$ satisfying $\phi= 0$ 
on $\Gamma {\times}(-\frac12,\frac12)$ we can choose 
$\phi_k(x):=\phi(x',x_3-w_k(x'))$ as test function in \eqref{E2-part}. Passing to
the limit with respect to $k$, we obtain that equation \eqref{E2-partlin} holds true also for $\alpha=3$. 

From \eqref{E2-partlin} it follows that $Ee_3=0$ a.e.\ in $\Omega$.
This property, together with the fact that $E$ is symmetric, entails 
\begin{equation}\label{E11}
E = \left(
\begin{array}{ccc}
E_{11} &E_{12}  &0\\
E_{12}  &E_{22} &0\\
0 &0 &0
\end{array}
\right)
\end{equation}
for any $\alpha\geq3$.

\medskip
\textit{Step 5. Zeroth moment of the Euler-Lagrange equations.}
Let $\bar E\colon S\to\mthree$ be the zeroth moment of the limit stress $E$, defined as
\begin{equation}\label{zeroth}
\bar{E}(x'):= \int_{-\frac{1}{2}}^{\frac{1}{2}} E(x)\,dx_3
\end{equation}
for every $x' \in S$. In the following we derive the equation satisfied 
by $\bar E$.

We consider as test function in \eqref{EuleroBEh} a map independent of the variable $x_3$.
More precisely, let $\psi\in C_b^1(\R^2;\R^2)$ be such that $\psi(x') = 0$ 
for every $x'\in\Gamma$. Choosing $\phi(x)=(\psi(x'),0)$ in \eqref{EuleroBEh}, we have 
\begin{equation}\label{Eetest2}
\int_{\Omega}[R^h E^h (R^h)^T]''{\,:\,}\nabla'\psi((y^h)')\, dx 
=0,
\end{equation}
where $[R^h E^h (R^h)^T]''$ denotes the $2{\times}2$ submatrix of $R^h E^h (R^h)^T$,
whose entries are given by 
$[R^h E^h (R^h)^T]''_{ij}=R^h E^h (R^h)^Te_i{\,\cdot\,}e_j$,
$1\leq i,j\leq2$.

As in the previous step, it is convenient to split the integral in \eqref{Eetest2} as 
\begin{align}\label{split0}
\int_{\Omega}[R^h E^h (R^h)^T]''{\,:\,}\nabla'\psi((y^h)')\, dx
& = 
\int_{\Omega}\chi_h[R^h E^h (R^h)^T]''{\,:\,}\nabla'\psi((y^h)')\, dx
\nonumber\\
{}+ 
& \int_{\Omega}(1-\chi_h)[R^h E^h (R^h)^T]''{\,:\,}\nabla'\psi((y^h)')\, dx.
\end{align} 
By \eqref{yh-conv} and the continuity and boundedness of $\nabla'\psi$, the sequence 
$\nabla'\psi((y^h)')$ converges to $\nabla'\psi$ strongly in $L^2(\Omega;\mathbb{M}^{2\times2})$. 
Thus, by \eqref{conL2} we obtain 
$$
\lim_{h\to0} \int_{\Omega}\chi_h[R^h E^h (R^h)^T]''{\,:\,}\nabla'\psi((y^h)')\, dx
= \int_{\Omega} E''{\,:\,} \nabla'\psi \,dx,
$$
while, using the boundedness of $\nabla'\psi$ and \eqref{conL1}, we have that the last integral
in \eqref{split0} converges to $0$, as $h\to0$.
Therefore, by \eqref{Eetest2} we conclude that
$$
\int_{\Omega} E''{\,:\,} \nabla'\psi \,dx = 0
$$
for every $\psi\in C_b^1(\R^2;\R^2)$ such that $\psi|_\Gamma=0$. 
In terms of the zeroth moment of the stress defined in \eqref{zeroth}, the previous equation yields
\begin{equation}\label{finale0}
\int_S \bar{E}''{\,:\,} \nabla'\psi \,dx' = 0
\end{equation}
for every $\psi\in C_b^1(\R^2;\R^2)$ such that $\psi|_\Gamma = 0$, and by approximation
for every $\psi\in H^1(S;\R^2)$ with $\psi|_\Gamma = 0$.

\medskip
\textit{Step 6. First moment of the Euler-Lagrange equations.}
We now derive the equation satisfied by the first moment of the stress, that is defined as
\begin{equation}\label{first-m}
\hat{E}(x'):= \int_{-\frac{1}{2}}^{\frac{1}{2}}x_3 E(x)\,dx_3
\end{equation}
for every $x' \in S$. 

Let $\varphi\in C^1_b(\R^2)$ be such that $\varphi|_\Gamma=0$ and let us consider 
$\phi(x)=(0,\frac1h \varphi(x'))$ in \eqref{EuleroBEh}.
Since \eqref{yh-conv} and the continuity and boundedness of $\varphi$ entail
$$
\lim_{h\to0}\int_{\Omega} g \varphi((y^h)')\,dx = 
\int_{\Omega} g\varphi\,dx = \int_S g\varphi\,dx',
$$
we deduce that
\begin{equation}\label{Eetest2h}
\lim_{h\to0} \ \int_{\Omega}\frac{1}{h}\sum_{i=1}^2 [R^h E^h (R^h)^T]_{3i}\, 
\partial_i\varphi((y^h)')\, dx 
= \int_S g\varphi \,dx'.
\end{equation}

On the other hand, we can make partially explicit the limit on the left-hand side by applying
the usual splitting $\Omega=B_h\cup (\Omega\setminus B_h)$ and considering the following decomposition:
\begin{equation}\label{Ehspezzato}
\frac{1}{h}R^h E^h (R^h)^T = h^{\alpha-3}A^hE^h (R^h)^T + 
h^{\alpha-3}E^h (A^h)^T + \frac{1}{h}E^h,
\end{equation}
where $A^h$ is the sequence introduced in \eqref{RId}. 
By the pointwise estimate \eqref{estEh} on $E^h$
and the fact that $|A^h|\leq C/h^{\alpha-2}$, we have
$$
|A^hE^h(R^h)^T +E^h(A^h)^T|\leq C \Big(\frac{W(\Id+h^{\alpha-1}G^h)}{h^{2\alpha-3}}+ |A^h||G^h|\Big),
$$
so that, from the energy bound \eqref{growth} we infer that
\begin{eqnarray*}
\lefteqn{\int_\Omega (1-\chi_h)|A^hE^h(R^h)^T +E^h(A^h)^T |\, dx}
\\
& \leq & C\int_{\Omega}\frac{W(\Id+h^{\alpha-1}G^h)}{h^{2\alpha-3}}\,dx + \|G^h\|_{L^2}\|(1-\chi_h)A^h\|_{L^2}
\\
& \leq & Ch+C\|(1-\chi_h)A^h\|_{L^2}. \vphantom{\int}
\end{eqnarray*}
As $A^h$ converges strongly in $L^2(S;\mthree)$, we conclude that
\begin{equation}\label{L1comp}
(1-\chi_h)\big(A^hE^h(R^h)^T +E^h(A^h)^T\big)\to 0 \quad \text{strongly in } L^1(\Omega;\mthree).
\end{equation}
On the other hand, by \eqref{RId} we have that
$A^h\to A$ strongly in $L^p(S;\mthree)$ for any $p<\infty$, while 
$\chi_hE^h\wto E$ weakly in $L^2(\Omega;\mthree)$ by \eqref{conL2}. Therefore,
\begin{equation}\label{convcomp}
\chi_h\big(A^hE^h (R^h)^T+ E^h (A^h)^T\big)\wto AE+EA^T \quad \text{weakly in } L^q(\Omega;\mthree)
\end{equation}
for any $q>2$. Using \eqref{L1comp}, \eqref{convcomp}, and the fact that
$\partial_i\varphi((y^h)')\to\partial_i\varphi$ strongly in $L^p(\Omega)$ 
for any $p<\infty$, we deduce that
\begin{align}
\lim_{h\to0} \ \int_{\Omega}\sum_{i=1}^2 
[A^hE^h (R^h)^T & +E^h(A^h)^T]_{3i}\partial_i\varphi((y^h)')\, dx
\nonumber
\\
& = \int_{\Omega}\sum_{i=1}^2 [AE+EA^T]_{3i}\partial_i\varphi \,dx.
\label{Eh1}
\end{align}
In particular, from the definition \eqref{RId} of $A$ and the structure property \eqref{E11} of $E$ 
we have that $[AE]_{3i}=\sum_{j=1}^2E_{ji}\partial_j v$ 
and $[EA^T]_{3i}=0$. Therefore, integrating over $x_3$ we obtain
\begin{equation}\label{eq-bar}
\int_{\Omega}\sum_{i=1}^2 [AE+EA^T]_{3i}\partial_i\varphi \,dx
= \int_S \bar{E}''{\,:\,}(\nabla'v{\,\otimes\,}\nabla'\varphi)\, dx'.
\end{equation}
Combining \eqref{Eetest2h}, \eqref{Ehspezzato}, \eqref{Eh1}, and \eqref{eq-bar}, we conclude that, if $\alpha=3$,
\begin{equation}\label{limitEE}
\lim_{h\to0} \ \int_{\Omega}\frac{1}{h}\sum_{i=1}^2 E^h_{3i}\partial_i\varphi((y^h)')\, dx 
= \int_S g\varphi \,dx' - \int_S \bar{E}''{\,:\,}(\nabla'v{\,\otimes\,}\nabla'\varphi)\, dx'
\end{equation}
for every $\varphi\in C_b^1(\R^2)$ such that $\varphi|_\Gamma=0$. 

If $\alpha>3$, by \eqref{Eh1} we deduce that
\begin{equation}\label{Eh21lin}
\lim_{h\to0} \ \int_{\Omega}\sum_{i=1}^2 h^{\alpha-3}[A^hE^h (R^h)^T +E^h(A^h)^T]_{3i}
\partial_i\varphi((y^h)')\, dx =0.
\end{equation}
Combining this equation with \eqref{Eetest2h} and \eqref{Ehspezzato}, 
we conclude that, if $\alpha>3$,
\begin{equation}\label{limitEElin}
\lim_{h\to0} \ \int_{\Omega}\frac{1}{h}\sum_{i=1}^2 E^h_{3i}\partial_i\varphi((y^h)')\, dx 
= \int_S g\varphi \,dx'
\end{equation}
for every $\varphi\in C_b^1(\R^2)$ such that $\varphi|_\Gamma=0$. 

We now want to identify the limit in \eqref{limitEE} and \eqref{limitEElin} in terms of the first moment~$\hat E$. 
This will be done, as in \cite{MP06}, by first considering in the Euler-Lagrange equations \eqref{EuleroBEh}
a suitable test function $\phi$ with a linear behaviour in $x_3$, and then passing to the limit with respect to $h$. 
Using the symmetry of $E$, this provides us with an identity (see \eqref{limitEq} below) relating the first moment $\hat E$ with the limit in \eqref{limitEE} (or \eqref{limitEElin})
and, by comparison with \eqref{limitEE} (or \eqref{limitEElin}, respectively) 
the limiting equation for $\hat E$ (see \eqref{finale} and \eqref{finalelin} below).

An additional difficulty with respect to \cite{MP06} is due to the fact that
the simple choice $\phi(x) = (x_3\eta(x'),0)$ is not allowed in our framework, 
as this test function is not bounded in $\R^3$. 
For this reason we introduce a truncation function $\theta^h$, which coincides with the identity in an interval $(-\omega_h,\omega_h)$, for a suitable $\omega_h\to+\infty$, and we consider a test function of the form \eqref{test3} below. The rate of convergence of $\omega_h$ has to be chosen in such a way to match two requirements.
On one hand, we need to show that the limiting contribution due to the region where $\theta^h$ does not
coincide with the identity is negligible. This can be done by means of the estimate \eqref{estL1-v}, once
we prove that the measure of the set $D_h$ where $|y^h_3/h|\geq \omega_h$ is sufficiently small. 
This is guaranteed if the rate of convergence of $\omega_h$ is fast enough (see proof of \eqref{term22} 
below).
On the other hand, because of this choice, the $L^\infty$-norm of the test functions $\phi^h$ is not bounded, but
blows up as $\omega_h$. Therefore, the convergence rate of $\omega_h$ has to be carefully chosen
to ensure that the integral on $\Omega\setminus B_h$ remains irrelevant, as usual.
This is possible owing to the choice of $B_h$ and the estimate \eqref{estL1} (see proof of \eqref{crucis2}
below).

To be definite, let $\omega_h$ be a sequence of positive numbers such that
\begin{equation}\label{omegah}
h\,\omega_h\to\infty, \qquad h^{\alpha-1-\gamma}\omega_h\to0,
\end{equation}
where $\gamma$ is the exponent introduced in the definition of $B_h$. This is possible
since $\gamma<\alpha-2$ (for instance, one can choose $\omega_h:=h^{-(\alpha-\gamma)/2}$).
Let $\theta^h\in C_b^1(\R)$ be a truncation function satisfying
\begin{align}
\label{barc}
& \theta^h(t) = t \quad \mbox{for } |t| \leq \omega_h,
\\
& |\theta^h(t)|\leq |t| \quad \mbox{for every } t\in\R, \vphantom{\frac{d\theta^h}{dt}}
\label{barc1}
\\
& \|\theta^h\|_{L^\infty}\leq 2\omega_h, \quad \Big\|\frac{d\theta^h}{dt}\Big\|_{L^\infty}\leq 2.
\label{barc2}
\end{align}
Let $\eta\in C^1_b(\R^2;\R^2)$ be such that $\eta(x')=0$ for every $x' \in \Gamma$. 
We define $\phi^h\colon \R^3 \to \R^3$ as 
\begin{equation}\label{test3}
\phi^h(x):= \left(\theta^h\Big(\frac{x_3}{h}\Big)\eta(x'), 0\right).
\end{equation}
Owing to the assumptions on $\theta^h$ and $\eta$, the $\phi^h$ are admissible test 
functions in \eqref{EuleroBEh}; then inserting $\phi^h$ in \eqref{EuleroBEh} leads to
\begin{multline}\label{Eetest3}
\int_{\Omega}\theta^h\Big(\frac{y_3^h}{h}\Big)[R^h E^h (R^h)^T]''{\,:\,}\nabla'\eta((y^h)')\,dx 
\\
{}+\int_{\Omega}\frac{1}{h}\sum_{i=1}^2 [R^h E^h (R^h)^T]_{i3}\, \eta_i((y^h)')
\Big(\frac{d\theta^h}{dt}\!\Big(\frac{y_3^h}{h}\Big)\Big)\,dx=0.
\end{multline}
We now compute the limit of each term in \eqref{Eetest3} separately, starting with the first.
We consider the usual splitting $\Omega=B_h\cup (\Omega\setminus B_h)$ and we carefully analyse the 
contributions of the integral in the two subdomains. 

If $\alpha=3$,  we have that
\begin{equation}\label{firstOhc}
\lim_{h\to0} \int_{\Omega}\chi_h\theta^h\Big(\frac{y_3^h}{h}\Big)[R^h E^h (R^h)^T]''
{\,:\,}\nabla'\eta((y^h)')\,dx  
= \int_S (\hat{E}''+v\bar{E}''){\,:\,}\nabla'\eta \,dx'.
\end{equation}
Indeed, by \eqref{y2-est} and \eqref{convV} the sequence $y_3^h/h$ converges to $x_3+v$ a.e.\ in $\Omega$ and
is dominated by an $L^2$ function. From \eqref{barc} and \eqref{barc1} it follows
that the sequence $\theta^h(y_3^h/h)$ converges to $x_3+v$ a.e.\ in $\Omega$ and
is dominated by an $L^2$ function. Owing to the convergence \eqref{yh-conv} of $y^h$ and to the continuity and 
boundedness of $\nabla'\eta$, we conclude that 
$$
\theta^h\Big(\frac{y_3^h}{h}\Big)\nabla'\eta((y^h)') \to (x_3 + v) \nabla'\eta(x') 
\quad \mbox{strongly in } L^2(\Omega; \R^2). 
$$
Therefore, by \eqref{conL2} we deduce 
$$
\lim_{h\to0} \int_{\Omega}\chi_h\theta^h\Big(\frac{y_3^h}{h}\Big)[R^hE^h (R^h)^T]''{\,:\,}\nabla'\eta((y^h)')\,dx  = \int_{\Omega} (x_3 + v)E''{\,:\,}\nabla'\eta(x') \,dx.
$$
Integration with respect to $x_3$ yields \eqref{firstOhc}.

As for the integral on $\Omega\setminus B_h$, by the estimate \eqref{barc2} on $\theta^h$ and
\eqref{estL1} it can be bounded by
\begin{align}
\int_{\Omega}(1-\chi_h)
\Big|\theta^h &
\Big(\frac{y_3^h}{h}\Big)[R^hE^h (R^h)^T]''{\,:\,}\nabla'\eta((y^h)')\Big|\,dx
\nonumber
\\
& \leq \hspace{2mm} 2\omega_h\|\nabla'\eta\|_{L^\infty}\int_{\Omega\setminus B_h} |E^h| \hspace{2mm}
\leq \hspace{2mm} C h^{\alpha-1-\gamma}\omega_h;
\label{crucis2}
\end{align}
therefore, it is infinitesimal as $h\to0$ by the second property in \eqref{omegah}.
We conclude that, if $\alpha=3$,
\begin{equation}\label{final1}
\lim_{h\to0} \int_{\Omega}\theta^h\Big(\frac{y_3^h}{h}\Big)[R^h E^h (R^h)^T]'' {\,:\,} \nabla'\eta((y^h)')\,dx  
= \int_S (\hat{E}''+v\bar{E}''){\,:\,} \nabla'\eta \,dx'.
\end{equation}

Analogously, for $\alpha>3$, since $y_3^h/h$ converges to $x_3$ strongly in $L^2(\Omega)$, we deduce that
\begin{equation}\label{final1lin}
\lim_{h\to0} \int_{\Omega}\theta^h\Big(\frac{y_3^h}{h}\Big)[R^h E^h (R^h)^T]'' {\,:\,} \nabla'\eta((y^h)')\,dx  
= \int_S \hat{E}''{\,:\,} \nabla'\eta \,dx'.
\end{equation}

In order to analyse the second integral in \eqref{Eetest3}, it is convenient to 
split it as follows:
\begin{align}
\int_{\Omega}\frac{1}{h} \sum_{i=1}^2 & [R^h E^h (R^h)^T]_{i3}\, \eta_i((y^h)')
\Big(\frac{d\theta^h}{dt}\!\Big(\frac{y_3^h}{h}\Big)\Big)\,dx 
\nonumber
\\
= \int_{\Omega}&\,\frac{1}{h} \sum_{i=1}^2 [R^h E^h (R^h)^T]_{i3}\, \eta_i((y^h)') \,dx
\nonumber
\\
& {}+ \int_{\Omega}\frac{1}{h} \sum_{i=1}^2 [R^h E^h (R^h)^T]_{i3}\, \eta_i((y^h)')
\Big(\frac{d\theta^h}{dt}\!\Big(\frac{y_3^h}{h}\Big)-1\Big)\,dx.
\label{split23}
\end{align} 
We claim that the second term on the right-hand side is infinitesimal, as $h\to0$, that is,
\begin{equation}\label{term22}
\lim_{h\to 0} \int_{\Omega}\frac{1}{h} \sum_{i=1}^2 [R^h E^h (R^h)^T]_{i3}\, \eta_i((y^h)')
\Big(\frac{d\theta^h}{dt}\!\Big(\frac{y_3^h}{h}\Big)-1\Big)\,dx= 0.
\end{equation}

If $\alpha=3$, combining \eqref{Eetest3}, \eqref{final1}, and \eqref{split23}, the claim implies that
\begin{equation}\label{new-eq}
\lim_{h\to0} \int_{\Omega}\frac{1}{h} \sum_{i=1}^2 [R^h E^h (R^h)^T]_{i3}\, \eta_i((y^h)') \,dx
= - \int_S (\hat{E}''+v\bar{E}''){\,:\,} \nabla'\eta \,dx'.
\end{equation}
On the other hand, we can identify the limit on the left-hand side of the equation above 
by applying the decomposition \eqref{Ehspezzato} and using the convergence properties 
\eqref{L1comp} and \eqref{convcomp}; in this way, we deduce that, if $\alpha=3$,
\begin{eqnarray}
\lefteqn{\lim_{h\to0} \int_{\Omega}\frac{1}{h} \sum_{i=1}^2 [R^h E^h (R^h)^T]_{i3}\, \eta_i((y^h)') \,dx}
\nonumber
\\
\label{term21}
& = & \int_S \bar{E}''{\,:\,}(\eta{\,\otimes\,}\nabla'v)\, dx'
\ + \ \lim_{h\to0} \int_{\Omega} \frac{1}{h}\sum_{i=1}^2 E^h_{i3}\, \eta_i((y^h)')\,dx.
\end{eqnarray}
Comparing \eqref{new-eq} and \eqref{term21}, and applying equation
\eqref{finale0} with $\psi=v\eta$, we conclude that, if $\alpha=3$,
\begin{equation}\label{limitEq}
\lim_{h\to0} \int_{\Omega} \frac{1}{h}\sum_{i=1}^2\, E^h_{i3}\, \eta_i((y^h)')\,dx = 
- \int_S \hat{E}''{\,:\,} \nabla'\eta \,dx'.
\end{equation}

If $\alpha>3$, combining \eqref{Eetest3}, \eqref{final1lin}, \eqref{split23}, and the claim
\eqref{term22}, we obtain
\begin{equation}\label{new-eqlin}
\lim_{h\to0} \int_{\Omega}\frac{1}{h} \sum_{i=1}^2 [R^h E^h (R^h)^T]_{i3}\, \eta_i((y^h)') \,dx
= - \int_S \hat{E}''{\,:\,} \nabla'\eta \,dx';
\end{equation}
moreover, by the decomposition \eqref{Ehspezzato} and the convergence properties \eqref{L1comp} and \eqref{convcomp}
we have
\begin{equation}\label{term21lin}
\lim_{h\to0} \int_{\Omega}\frac{1}{h} \sum_{i=1}^2 [R^h E^h (R^h)^T]_{i3}\, \eta_i((y^h)') \,dx= \lim_{h\to0} \int_{\Omega} \frac{1}{h}\sum_{i=1}^2 E^h_{i3}\, \eta_i((y^h)')\,dx.
\end{equation}
Comparing \eqref{new-eqlin} and \eqref{term21lin}, we deduce that equation \eqref{limitEq}
is satisfied also for $\alpha>3$.

It remains to prove \eqref{term22}. To this aim we introduce the set
$D_h:=\{x\in\Omega: |y^h_3(x)|/h \geq \omega_h \}$. Since
the sequence $y_3^h/h$ is bounded in $L^2(\Omega)$
by \eqref{y2-est} and \eqref{convV}, we have
$$
|D_h| \leq \omega_h^{-1} \int_{D_h} \frac{|y_3^h|}{h}\,dx 
\leq c\, \omega_h^{-1} |D_h|^{1/2},
$$
which implies
\begin{equation}\label{setest}
|D_h| \leq C\omega_h^{-2}.
\end{equation}

Since the derivative of $\theta^h$ is equal to $1$ on $(-\omega_h, \omega_h)$ by \eqref{barc}, 
the integral in \eqref{term22} reduces to
\begin{multline}\nonumber
\int_{\Omega}\frac{1}{h} \sum_{i=1}^2 [R^h E^h (R^h)^T]_{i3}\, \eta_i((y^h)')
\Big(\frac{d\theta^h}{dt}\!\Big(\frac{y_3^h}{h}\Big)-1\Big)\,dx
\\
= \int_{D_h}\frac{1}{h} \sum_{i=1}^2 [R^h E^h (R^h)^T]_{i3}\, \eta_i((y^h)')
\Big(\frac{d\theta^h}{dt}\!\Big(\frac{y_3^h}{h}\Big)-1\Big)\,dx.
\end{multline}
By \eqref{estL1-v}, \eqref{barc2}, and \eqref{setest}, we have
\begin{align*}
\Big| \int_{D_h}\frac{1}{h}  \sum_{i=1}^2 & [R^h E^h (R^h)^T]_{i3}\, \eta_i((y^h)')
\Big(\frac{d\theta^h}{dt}\!\Big(\frac{y_3^h}{h}\Big)-1\Big)\,dx
 \Big| 
\\
& \leq  \frac{C}{h}\Big(1 + \Big\|\frac{d\theta^h}{dt}\Big\|_{L^\infty}\Big)\|\eta \|_{L^\infty}
\int_{D_h} |E^h| \,dx 
\\
& \leq  Ch^{\alpha-2}+ \frac{C}{h}|D_h|^{1/2} \leq Ch^{\alpha-2} +\frac{C}{h\,\omega_h}.
\end{align*}
By \eqref{omegah} this proves the claim \eqref{term22}.

\medskip
\textit{Step 7. Limit equations.} Let $\varphi \in C^2_b(\R^2)$ be such that $\varphi(x') = 0$, 
$\nabla'\varphi(x') = 0$ for every $x'\in \Gamma$. Since $E$ is symmetric,
we can compare equation \eqref{limitEq} (where we specify $\eta = \nabla'\varphi$) with
\eqref{limitEE}, if $\alpha=3$, or \eqref{limitEElin}, if $\alpha>3$. In this way we deduce
that, if $\alpha=3$
\begin{equation}\label{finale}
\int_S \bar{E}''{\,:\,}(\nabla'v{\,\otimes\,}\nabla'\varphi)\, dx'
- \int_S \hat{E}''{\,:\,} (\nabla')^2\varphi \,dx' = \int_S g \varphi\,dx',
\end{equation} 
while, if $\alpha>3$,
\begin{equation}\label{finalelin}
- \int_S \hat{E}''{\,:\,} (\nabla')^2\varphi \,dx' = \int_S g \varphi\,dx'.
\end{equation} 
By approximation the two equations hold for every $\varphi\in H^2(S)$ with $\varphi|_\Gamma=0$
and $\nabla'\varphi|_\Gamma=0$.

In order to express the limiting equations \eqref{finale0}, \eqref{finale}, and \eqref{finalelin} in terms of the limit displacements, 
an explicit characterization of $\bar E''$ and $\hat{E}''$ is needed. 
Since $E = \leb G$ and $E$ is of the form \eqref{E11}, 
we have $E'' = \leb_2 G''$ (see \cite[Proposition~3.2]{MP06}). 
Therefore, by \eqref{G-id1} and \eqref{G-id2} we obtain, for $\alpha=3$,
$$
\bar E''= \leb_2 \big(\sym\nabla'u + \frac12 \nabla'v \otimes \nabla'v\big),
\quad \hat E''=-\frac{1}{12}\leb_2 (\nabla')^2 v.
$$
These identities, together with \eqref{finale0} and \eqref{finale}, provide us with the Euler-Lagrange equations
\eqref{EulLag-u}--\eqref{EulLag-v}. 

By \eqref{G-id1} and \eqref{G-id2lin} we obtain, for $\alpha>3$,
$$
\bar E''= \leb_2(\sym\nabla'u),
\quad \hat E''=-\frac{1}{12}\leb_2 (\nabla')^2 v.
$$
The first identity, together with \eqref{finale0} and the boundary condition $u=0$ on $\Gamma$, implies
that $u=0$, while the second identity, together with \eqref{final1lin}, provide us with the Euler-Lagrange
equation \eqref{EulLag-lin}.
This concludes the proof.
\end{proof}

\bigskip
\bigskip
\medskip

\noindent
\textbf{Acknowledgements.}
This work is part of the project ``Problemi di riduzione di dimensione per strutture elastiche sottili'' 2008, supported by GNAMPA. M.G.M. was also partially supported by MiUR through the project ``Variational problems with multiple scales'' 2006. L.S. was partially supported by  Marie Curie Research Training Network MRTN-CT-2004-505226 (MULTIMAT).

\end{document}